\documentclass[11pt]{amsart}
\usepackage[T1]{fontenc}
\usepackage{lmodern}
\usepackage[margin=1.15in]{geometry}
\usepackage{amsmath,amssymb,amsthm,booktabs,microtype}
\usepackage[hidelinks]{hyperref}
\hypersetup{pdftitle={Persistent Hilbert-function defects for products of powers of linear forms},
 pdfauthor={Boris Shapiro},pdfkeywords={Hilbert functions, power ideals, fat points, Waldschmidt constants}}

\newtheorem{theorem}{Theorem}[section]
\newtheorem{proposition}[theorem]{Proposition}
\newtheorem{corollary}[theorem]{Corollary}
\newtheorem{lemma}[theorem]{Lemma}
\theoremstyle{definition}

\newtheorem{question}[theorem]{Question}
\newcommand{\CC}{\mathbb C}

\newcommand{\PP}{\mathbb P}
\newcommand{\HF}{\operatorname{HF}}
\newcommand{\HS}{\operatorname{HS}}
\newcommand{\rank}{\operatorname{rank}}

\title[Generic $\mu$-power ideals, fat points, and symbolic powers]{Generic $\mu$-power ideals, fat points, and symbolic powers}
\author{Boris Shapiro}
\address{Department of Mathematics, Stockholm University, SE-106 91 Stockholm, Sweden}
\email{shapiro@math.su.se}
\date{}

\begin{document}
\begin{abstract}
Multiplying powers of general linear forms by additional factors need not restore
the Hilbert series predicted for general forms. We give a quantitative criterion
for this failure in terms of a hypersurface with prescribed multiplicities at
the dual points. Under a strict degree--multiplicity inequality, the failure
occurs in every sufficiently large degree, even when the degrees of the additional
factors grow linearly with the degree of the generators. Moreover, the quotient
has dimension at least a constant times $d^{n-1}$ in a linearly growing interval
where the predicted Hilbert function is zero. The argument applies to arbitrary
additional factors, which may have different degrees on different generators.
The Waldschmidt constant determines the sharp asymptotic threshold for this criterion.
Conics and related plane curves give examples with five, six, seven, and eight
generators in three variables; a quadric gives examples with nine generators in
four variables. For five ternary generators the first non-pure failure occurs
in degree $14$, for the partition $(13,1)$, and its Hilbert series is exactly the
predicted series plus $t^{24}$. The obstruction is geometric; the assertion of
minimality is supported by a reproducible exact certificate.
\end{abstract}
\maketitle

\section{Introduction}

Let $S=\CC[x_1,\ldots,x_n]$, where $n\ge2$, and write
\[
 G_{n,d,r}(t)=\left[\frac{(1-t^d)^r}{(1-t)^n}\right]_+.
\]
Here the brackets mean truncation at the first non-positive coefficient.
Fr\"oberg's conjecture predicts this Hilbert series for the quotient by $r$
general forms of degree $d$ \cite{Fro85}. In three variables the prediction is a
theorem of Anick \cite{An86}.

For a partition $\mu=(\mu_1,\ldots,\mu_s)\vdash d$, a $\mu$-power form is a product
$L_1^{\mu_1}\cdots L_s^{\mu_s}$ of pairwise non-proportional linear forms.
Problem F of Fr\"oberg, Lundqvist, Oneto, and Shapiro \cite{FLOS18} asks whether
$r$ general such forms have Hilbert series $G_{n,d,r}$ whenever $\mu\ne(d)$.
The exclusion of the pure partition is necessary: ideals generated by powers
of linear forms can have larger Hilbert functions, as reflected in the
Fr\"oberg--Iarrobino conjecture \cite{Iar97}.

The answer to Problem F is negative. More specifically, a large linear-factor
multiplicity can force a defect that survives the introduction of additional
factors. Suppose that
\begin{equation}\label{eq:generators}
 I=(L_1^{d-a_1}H_1,\ldots,L_r^{d-a_r}H_r),\qquad \deg H_i=a_i.
\end{equation}
The elementary containment
$I\subseteq(L_1^{d-a_1},\ldots,L_r^{d-a_r})$, together with the inverse-system
description of powers of linear forms \cite{EI95}, gives lower bounds for
$\HF_{S/I}$. The issue is whether those bounds remain larger than the predicted
Hilbert function after the degrees of the generators have increased. The main
result gives a uniform answer, including when the $a_i$ increase with $d$.

Let $T=\CC[X_1,\ldots,X_n]$ be the dual polynomial ring. A linear form
$L_i=\sum_j\lambda_{ij}x_j$ determines the point
$P_i=[\lambda_{i1}:\cdots:\lambda_{in}]$ in the projective space with coordinates
$X_1,\ldots,X_n$. Put $N=n-1$.

\begin{theorem}[Persistence under additional factors]\label{thm:persistence}
Suppose that a nonzero form $B\in T_k$ vanishes to order at least $m_0\ge1$ at
each of $P_1,\ldots,P_r$, and set
\[
 \sigma=\frac{k}{m_0},\qquad R=r^{1/N}.
\]
Assume $2\le\sigma<R$. For every real number
\begin{equation}\label{eq:epsilon}
 0\le\varepsilon<\frac{R-\sigma}{\sigma(R-1)}
\end{equation}
there exist constants $d_0$, $c>0$, and $1<\lambda_-<\lambda_+<2$ with the
following property. For every integer $d\ge d_0$, every choice of integers
$0\le a_i\le\varepsilon d$, and every choice of nonzero $H_i\in S_{a_i}$, the
ideal \eqref{eq:generators} satisfies
\begin{equation}\label{eq:large-defect}
 [t^q]G_{n,d,r}=0,\qquad \HF_{S/I}(q)\ge c d^{n-1}
 \quad\text{for }\lambda_-d\le q\le\lambda_+d.
\end{equation}
In particular, the sum of these Hilbert-function values over the indicated
integer degrees is at least $c'd^n$ for some $c'>0$ and all sufficiently large $d$.
\end{theorem}

The constants depend only on the numerical data and $\varepsilon$; no
generality assumption on the $H_i$ is needed. In particular, they may be
products of linear forms with prescribed multiplicities. A fixed bound on
the $a_i$ is allowed, since it is eventually smaller than $\varepsilon d$ for
any fixed positive $\varepsilon$ satisfying \eqref{eq:epsilon}.

The required form $B$ is easy to construct in the following cases. The proof
in Section~\ref{sec:examples} uses only interpolation and multiplication of
forms; the familiar interpretation by curves on del Pezzo surfaces is not
needed for existence.

\begin{corollary}\label{cor:mainfamilies}
For each of
\[
 (n,r)=(3,5),(3,6),(3,7),(3,8),(4,9)
\]
and each fixed integer $a\ge1$, every non-pure partition of $d$ having a part at least
$d-a$ gives a negative answer to Problem F for all sufficiently large $d$.
The stronger conclusion of Theorem~\ref{thm:persistence}, allowing different
partitions on different generators, also holds in each case.
\end{corollary}

For $n=3$ these are failures of the actual Hilbert series of general forms,
by Anick's theorem. In four variables the comparison throughout the paper is
with $G_{4,d,r}$; no unproved case of Fr\"oberg's conjecture is used to identify
this series with the Hilbert series of general forms.

The first ternary example is quite small, although it lies beyond the range
suggested by low-degree experiments.

\begin{theorem}[The first five-generator failure]\label{thm:first}
Five general $\mu$-power forms in three variables have Hilbert series
$G_{3,d,5}$ for every non-pure partition of every $2\le d\le13$.
In degree $14$ the same assertion holds for every non-pure partition except
$(13,1)$. For five general forms of this exceptional type,
\[
 \HS_{S/(L_1^{13}M_1,\ldots,L_5^{13}M_5)}(t)
   =G_{3,14,5}(t)+t^{24}.
\]
\end{theorem}

The nonzero class in degree $24$ is provided by the twelfth power of the
conic through the five dual points. Its existence requires no computation.
The exact Hilbert series and the assertion that no earlier non-pure example
occurs use the finite certificate in Appendix~\ref{app:certificate}.

The surrounding positive results help delimit the problem. Coalescing the
linear factors shows that every positive result for pure powers also applies
to all partitions of the same degree. Stanley's strong Lefschetz theorem
therefore settles $r\le n+1$ for every partition. More generally, Nenashev's
maximal-rank results apply to any nonempty class of forms invariant under
linear changes of coordinates \cite{Nen17}, hence directly to $\mu$-power
forms. The equivariant vector-bundle methods of Blomenhofer and Casarotti
\cite{BC25} give further results of this kind. Recent work of Boij, Dannetun,
and Lundqvist \cite{BDL26} proves additional cases for general forms in low
degrees. The present results concern a complementary range: a fixed number
of generators and increasing degree, with a large repeated linear factor.

Section~\ref{sec:comparison} proves the comparison statements and
Theorem~\ref{thm:persistence}. Section~\ref{sec:examples} constructs the
families and gives explicit thresholds. Section~\ref{sec:boundary} proves
Theorem~\ref{thm:first} and explains what the method leaves open.

\section{Comparison with power ideals and persistence}\label{sec:comparison}

We first record the two different ways in which a product of powers can be
compared with a pure power. Specialization gives upper bounds on the generic
Hilbert function, whereas divisibility gives lower bounds for every member
of a family.

\subsection{Specialization and the positive range}

A partition $\nu$ is a coarsening of $\mu$ if its parts are obtained by
grouping and summing the parts of $\mu$.

\begin{proposition}\label{prop:coarsening}
Let $\mu\vdash d$ and let $\nu$ be a coarsening of $\mu$.
The generic Hilbert function for $r$ $\mu$-power generators is bounded above,
degree by degree, by that for $r$ $\nu$-power generators.
If the latter Hilbert series equals $G_{n,d,r}$, then so does the former.
The same assertion holds when the partition is allowed to vary from one
generator to another.
\end{proposition}

\begin{proof}
Coalesce the factors in each group to obtain a $\nu$-power form. Thus the
$\nu$-power locus is contained in the closure of the $\mu$-power locus.
In degree $q\ge d$, the quotient dimension is the corank of
\begin{equation}\label{eq:multiplication}
 \Phi_q:S_{q-d}^{\,r}\longrightarrow S_q,
 \qquad (A_1,\ldots,A_r)\longmapsto\sum_{i=1}^r A_iF_i.
\end{equation}
Matrix rank is lower semicontinuous, which proves the first assertion.
Here we use the usual generic Hilbert series, constant on a nonempty open
subset of the fixed-degree parameter space.
Fr\"oberg's general lower bound is lexicographic \cite{Fro85}. If the
specialized series is $G_{n,d,r}$, the generic series is coefficientwise at
most $G_{n,d,r}$ and lexicographically at least $G_{n,d,r}$, so equality
follows. The specialization can be made independently for each generator.
\end{proof}

\begin{corollary}\label{cor:small-r}
If $r\le n+1$, then general products of powers of linear forms, each of total
degree $d$, have Hilbert series $G_{n,d,r}$ for any prescribed partitions.
For binary forms the conclusion holds for every $r$.
\end{corollary}

\begin{proof}
Specialize each generator to a pure power. For $r\le n$ one obtains an
initial segment of $x_1^d,\ldots,x_n^d$, a regular sequence. For $r=n+1$ use
$x_1^d,\ldots,x_n^d,(x_1+\cdots+x_n)^d$ and Stanley's strong Lefschetz theorem
for monomial complete intersections \cite{Sta78}.
In two variables the inverse-system description reduces the pure-power
case to prescribed vanishing at distinct points of $\PP^1$; in degree
$q\ge d$ its dimension is $\max\{q+1-r(q-d+1),0\}$. This gives $G_{2,d,r}$.
Apply Proposition~\ref{prop:coarsening}.
\end{proof}

A failure for a coarsening does not, by itself, imply a failure for a
refinement. The next comparison supplies the lower bound needed for that
direction.

\subsection{Divisibility and inverse systems}

The ring $S$ acts on $T$ by constant-coefficient differentiation. For a
homogeneous ideal $J$, write $(S/J)_q^\vee$ for the orthogonal complement
of $J_q$ under the perfect pairing $S_q\times T_q\longrightarrow\CC$.
Let $\mathfrak p_i\subset T$ be the ideal of $P_i$. For
$Z=\{P_1,\ldots,P_r\}$, put
\[
 I_Z^{(m)}=\bigcap_i\mathfrak p_i^m,\qquad
 h_Z(q;m)=\dim[I_Z^{(m)}]_q\quad(m\ge1).
\]
We use the convention $h_Z(q;m)=\dim T_q$ when $m\le0$.

\begin{proposition}\label{prop:comparison}
Let $I=(L_1^{b_1}H_1,\ldots,L_r^{b_r}H_r)$, where $b_i\ge1$ and the $H_i$
are arbitrary homogeneous forms. In every degree $q$ there is an inclusion
\begin{equation}\label{eq:inclusion}
 \left[\bigcap_{i=1}^r\mathfrak p_i^{\,\max\{q-b_i+1,0\}}\right]_q
 \ \subseteq\ (S/I)_q^\vee.
\end{equation}
In particular, if all generators have degree $d$, $b_i=d-a_i$, and
$a=\max_i a_i$, then
\begin{equation}\label{eq:fat-bound}
 \HF_{S/I}(q)\ge h_Z(q;q-d+a+1).
\end{equation}
\end{proposition}

\begin{proof}
Set $J=(L_1^{b_1},\ldots,L_r^{b_r})$. Since $I\subseteq J$, we have
$(S/J)_q^\vee\subseteq(S/I)_q^\vee$.
The inverse-system identity of Emsalem and Iarrobino \cite{EI95} identifies
$(S/J)_q^\vee$ with the left side of \eqref{eq:inclusion}. For completeness,
take one factor $L=x_1$ and its dual point $P=[1:0:\cdots:0]$.
A degree-$q$ form is annihilated by $\partial_{X_1}^b$ precisely when every
one of its monomials has total degree at least $q-b+1$ in $X_2,\ldots,X_n$.
This is vanishing to that order at $P$. Intersecting these conditions proves
the identity. The uniform multiplicity in \eqref{eq:fat-bound} is at least
each of the multiplicities in \eqref{eq:inclusion}.
\end{proof}

This argument uses no condition on the factorization of $H_i$. It is also
valid when their degrees differ. The apolarity identity is classical; the
application here is the comparison between its vanishing conditions and
the cutoff for generators of the larger degree $d$.

For later use set
\begin{equation}\label{eq:coefficient}
 c_{n,d,r}(q)=\sum_{j=0}^r(-1)^j\binom rj
                    \binom{q-jd+n-1}{n-1},
\end{equation}
where $\binom{u}{v}=0$ for integers $u<v$ and $v\ge0$.

\begin{corollary}[A finite-degree criterion]\label{cor:criterion}
Let $q\ge d$ and $0\le a<d$. If
\[
 h_Z(q;q-d+a+1)>0\quad\text{and}\quad c_{n,d,r}(q)\le0,
\]
then every ideal \eqref{eq:generators} with $a_i\le a$ has
$\HF_{S/I}(q)>[t^q]G_{n,d,r}=0$.
When $q<2d$, the second condition is simply
\begin{equation}\label{eq:two-term}
 \binom{q+n-1}{n-1}\le r\binom{q-d+n-1}{n-1}.
\end{equation}
Equivalently, writing $m=q-d+a+1$, it says that the affine virtual dimension
of the system of degree-$q$ forms with multiplicity $m-a$ at the $r$ points
is non-positive.
\end{corollary}

\begin{proof}
Use \eqref{eq:fat-bound}. A non-positive coefficient of the untruncated series
forces the corresponding coefficient of its positive truncation to be zero.
For $q<2d$ only the terms $j=0,1$ in \eqref{eq:coefficient} remain.
\end{proof}

\subsection{The quantitative argument}

\begin{proof}[Proof of Theorem~\ref{thm:persistence}]
The inequality on $\varepsilon$ is equivalent to
\[
 \frac{R}{R-1}<\frac{\sigma(1-\varepsilon)}{\sigma-1}.
\]
Since $R>\sigma\ge2$, we may choose
\begin{equation}\label{eq:lambda-window}
 \frac{R}{R-1}<\lambda_-<\lambda_+
       <\frac{\sigma(1-\varepsilon)}{\sigma-1}\le2.
\end{equation}
Put $\delta=\sigma(1-\varepsilon)-(\sigma-1)\lambda_+>0$.
For an integer $q$ in the specified interval and $a=\max_i a_i$, let
\[
 m=q-d+a+1,\qquad t=\left\lceil\frac{m}{m_0}\right\rceil,
 \qquad u=q-kt.
\]
For all sufficiently large $d$ we have $m\ge1$, and
\[
 u\ge q-\sigma m-k
   =\sigma(d-a-1)-(\sigma-1)q-k
   \ge\delta d-\sigma-k\ge\frac{\delta d}{2}.
\]
Multiplication by $B^t$ is injective on $T_u$ and its image consists of
degree-$q$ forms vanishing to order at least $m$ at every $P_i$. Thus
Proposition~\ref{prop:comparison} gives
\begin{equation}\label{eq:multiples-bound}
 \HF_{S/I}(q)\ge\dim(B^tT_u)=\binom{u+N}{N}
                \ge\frac{\delta^N}{2^N N!}d^N.
\end{equation}

On the other hand, $d<q<2d$. Uniformly for $s=q/d$ in the compact interval
$[\lambda_-,\lambda_+]$,
\[
 c_{n,d,r}(q)=\frac{d^N}{N!}\bigl(s^N-r(s-1)^N\bigr)+O(d^{N-1}).
\]
The expression in parentheses is strictly negative there, since
$s>R/(R-1)$. Hence $c_{n,d,r}(q)<0$ for all sufficiently large $d$,
uniformly in $q$ and in the $a_i$. This proves \eqref{eq:large-defect}.
The interval contains at least $(\lambda_+-\lambda_-)d/2$ integers for large
$d$, which gives the last assertion.
\end{proof}

Multiplying $B^t$ by all forms of degree $u$ gives the quantitative
conclusion: a positive linear margin for $u$ provides both a growing
defect and the absence of a congruence restriction on $d$.

The geometric hypothesis can be expressed by the Waldschmidt constant
\[
 \widehat\alpha(Z)=\inf_{m\ge1}\frac{\alpha(I_Z^{(m)})}{m}
                 =\lim_{m\to\infty}\frac{\alpha(I_Z^{(m)})}{m},
\]
where $\alpha(J)$ is the least degree of a nonzero form of $J$.
The equality follows from subadditivity. This invariant measures the
asymptotic degree needed for uniform vanishing; see \cite{CHMR13} for its
relation with Nagata's conjecture.

\begin{corollary}\label{cor:waldschmidt}
If $r>2^{n-1}$ and $\widehat\alpha(Z)<r^{1/(n-1)}$, then the conclusion of
Theorem~\ref{thm:persistence} holds for some $\varepsilon>0$.
\end{corollary}

\begin{proof}
Choose $B$ with degree--multiplicity ratio strictly smaller than
$r^{1/(n-1)}$. If this ratio is smaller than $2$, multiply $B$ by a form
of degree $2m_0-\deg B$. The resulting ratio is $2$, still smaller than
$r^{1/(n-1)}$. Apply Theorem~\ref{thm:persistence}.
\end{proof}

The Waldschmidt constant also gives the asymptotic boundary of the
uniform fat-point criterion. This boundary concerns the criterion itself,
not every possible source of Hilbert-function defect.

\begin{proposition}\label{prop:optimal-criterion}
Suppose $2\le\omega=\widehat\alpha(Z)<R=r^{1/(n-1)}$, and put
\[
 \eta=\frac{R-\omega}{\omega(R-1)}.
\]
If $d_j\to\infty$ and triples $(d_j,a_j,q_j)$ satisfy the hypotheses of
Corollary~\ref{cor:criterion}, then
\[
 \limsup_{j\to\infty}a_j/d_j\le\eta.
\]
Conversely, for every $0\le\varepsilon<\eta$ and all sufficiently large $d$,
the criterion applies with $a=\lfloor\varepsilon d\rfloor$ and a suitable $q$.
\end{proposition}

\begin{proof}
For a triple in the first assertion, write $m=q-d+a+1$. Nonemptiness and
the definition of $\omega$ give $q\ge\omega m\ge2m$, so $d\le q<2d$ and
\[
 \frac ad\le1-\frac1d-\frac{\omega-1}{\omega}\frac qd.
\]
Choose a subsequence realizing the limit superior of $a/d$, and then a
further subsequence along which $q/d\to\lambda\in[1,2]$.
Dividing \eqref{eq:two-term} by $d^{n-1}$ and taking limits gives
$\lambda^{n-1}\le r(\lambda-1)^{n-1}$. Hence
$\lambda\ge R/(R-1)$, and substitution in the preceding inequality yields
the required bound.

For the converse, choose $B$ whose ratio $\sigma$ is sufficiently close
to $\omega$ that
$\varepsilon<(R-\sigma)/(\sigma(R-1))$.
The definition of the infimum ensures such a choice, with
$2\le\sigma<R$. The construction in the proof of
Theorem~\ref{thm:persistence} gives a nonzero uniform fat-point system
and a negative coefficient for every sufficiently large $d$.
\end{proof}

\section{Plane curves and a spatial quadric}\label{sec:examples}

The following construction suffices for all the applications. At a point
of $\PP^{n-1}$, multiplicity at least $v$ imposes at most
$\binom{v+n-2}{n-1}$ linear conditions on forms of a fixed degree.

\begin{lemma}\label{lem:products}
For each $i=1,\ldots,r$, suppose that there is a nonzero degree-$s$ form
$C_i$ with multiplicity at least $u$ at $P_i$ and at least $v$ at all the
other points. Then $B=\prod_i C_i$ has degree $rs$ and multiplicity at
least $u+(r-1)v$ at every point.
Such $C_i$ exist whenever
\[
 \binom{s+n-1}{n-1}>
 \binom{u+n-2}{n-1}+(r-1)\binom{v+n-2}{n-1},
\]
where the summand for multiplicity zero is understood to be zero.
\end{lemma}

\begin{proof}
The displayed inequality leaves a nonzero solution to the homogeneous
linear vanishing conditions. Multiplicities add under multiplication.
\end{proof}

For five points of $\PP^2$, take a conic through them. For six points,
multiply the six conics obtained by omitting one point in turn. For seven
points, take cubics with a double point at one specified point and simple
points at the other six. There are $10$ coefficients and at most $3+6=9$
conditions. For eight points, take sextics with multiplicities $3,2,\ldots,2$;
there are $28$ coefficients and at most $6+7\cdot3=27$ conditions. Multiplying
the seven cubics or eight sextics gives the third and fourth plane rows of
Table~\ref{tab:geometry}. Finally, a quadric in $\PP^3$ through nine points
exists because the space of quadrics has dimension $10$.

\begin{table}[ht]
\centering
\caption{Forms supplying the persistence criterion. The last column is
the strict upper bound for $\varepsilon$ in \eqref{eq:epsilon}, rounded
for orientation; the exact value is given by that formula.}
\label{tab:geometry}
\begin{tabular}{@{}ccrrcc@{}}
\toprule
$n$&$r$&$k$&$m_0$&$\sigma=k/m_0$&$\varepsilon_*$\\
\midrule
3&5&2&1&$2$&$0.095492$\\
3&6&12&5&$12/5$&$0.014226$\\
3&7&21&8&$21/8$&$0.004803$\\
3&8&48&17&$48/17$&$0.000949$\\
4&9&2&1&$2$&$0.037073$\\
\bottomrule
\end{tabular}
\end{table}

\begin{proof}[Proof of Corollary~\ref{cor:mainfamilies}]
All rows have $k\ge2m_0$. The strict inequalities needed in
Theorem~\ref{thm:persistence} are
\[
 2^2<5,\quad 12^2<6\cdot5^2,\quad 21^2<7\cdot8^2,
 \quad 48^2<8\cdot17^2,\quad 2^3<9.
\]
The constructions above work for general points, indeed for arbitrary
distinct points. For a partition having a part at least $d-a$, extract
that linear power and apply Theorem~\ref{thm:persistence} with a fixed
positive $\varepsilon$ and $d$ large enough that $a\le\varepsilon d$.
\end{proof}

For general plane points, the conics, cubics, and sextics in this construction
are the familiar curves of self-intersection $-1$ on the corresponding
blowups. The proof only requires a nonzero form with the indicated
multiplicities. In particular, it does not infer uniqueness or independence
of conditions from a dimension count.

\subsection{Explicit finite thresholds}

Although Theorem~\ref{thm:persistence} applies in every sufficiently large
degree, a single power of $B$ gives convenient explicit examples.

\begin{proposition}\label{prop:finitefamilies}
For any row of Table~\ref{tab:geometry}, fix $a\ge0$ and an integer $t\ge1$
with $m_0t>a$. Set
\[
 q=kt,\qquad d=(k-m_0)t+a+1.
\]
If
\begin{equation}\label{eq:threshold}
 \binom{kt+n-1}{n-1}\le r\binom{m_0t-a+n-2}{n-1},
\end{equation}
then every ideal generated by forms $L_i^{d-a}H_i$, with $\deg H_i=a$,
has positive Hilbert function in degree $q$, while $[t^q]G_{n,d,r}=0$.
\end{proposition}

\begin{proof}
The form $B^t$ has degree $q$ and multiplicity at least $m_0t$ at the
points. Moreover $q-d+a+1=m_0t$ and $q<2d$, since $k\ge2m_0$.
Apply Corollary~\ref{cor:criterion}.
\end{proof}

For $(n,r)=(3,5)$ the threshold has a particularly simple form.

\begin{corollary}\label{cor:conic}
Let $a\ge0$ and
\begin{equation}\label{eq:conicthreshold}
 d\ge a+1+
 \left\lceil\frac{10a+1+\sqrt{80a^2+40a+9}}{2}\right\rceil.
\end{equation}
Then any five forms $F_i=L_i^{d-a}H_i\in\CC[x,y,z]_d$, with $\deg H_i=a$,
satisfy
\[
 \HF_{S/(F_1,\ldots,F_5)}(2d-2a-2)\ge1,
 \qquad [t^{2d-2a-2}]G_{3,d,5}=0.
\]
\end{corollary}

\begin{proof}
Put $t=d-a-1$. Twice the difference between the two sides of
\eqref{eq:threshold} is
$-t^2+(10a+1)t+2-5a(a-1)$, which is non-positive under
\eqref{eq:conicthreshold}. The witness is $C^t$ for a conic $C$ through the
five coefficient points. Such a conic exists even if the points are not
in general position.
\end{proof}

For five general plane points, $\widehat\alpha(Z)=2$. The conic gives the
upper bound. If a curve of degree $q<2m$ vanished to order $m$ at the five
points, B\'ezout would force the conic to be a component. Repeatedly
removing it would give a curve of negative degree, a contradiction.
Consequently Proposition~\ref{prop:optimal-criterion} shows that
\[
 \eta=\frac{3-\sqrt5}{8}
\]
is the exact asymptotic boundary for this uniform fat-point criterion.
The explicit bound \eqref{eq:conicthreshold} has the same limiting ratio
$a/d$.

For $a=1$ the first values furnished by the chosen forms $B$ are as follows.
These are thresholds for this construction, not claims of minimality for
Problem F, except in the first row by Theorem~\ref{thm:first}.
\begin{center}
\begin{tabular}{@{}ccrrr@{}}
\toprule
$n$&$r$&$t$&$d$&$q$\\
\midrule
3&5&12&14&24\\
3&6&12&86&144\\
3&7&18&236&378\\
3&8&36&1118&1728\\
4&9&26&28&52\\
\bottomrule
\end{tabular}
\end{center}
Each row follows by substituting in \eqref{eq:threshold}; the accompanying
certificate also verifies the integer thresholds. In the spatial example,
$Q^{26}$ survives in degree $52$, and
\[
 \binom{55}{3}-9\binom{27}{3}=26235-26325=-90.
\]

\subsection{The limitation imposed by Nagata's inequality}

The same degree--multiplicity comparison explains why these particular
constructions stop at eight plane points. It does not classify all possible
failures of Problem F.

\begin{proposition}\label{prop:nagata}
For nine general points of $\PP^2$, the hypotheses of
Corollary~\ref{cor:criterion} cannot hold for any $a\ge0$.
For $r\ge10$ very general points, the same is true for $a\ge1$ if
$\alpha(I_Z^{(m)})\ge m\sqrt r$ for every $m\ge1$.
In particular, the latter conclusion follows from Nagata's conjecture,
and is unconditional when $r\ge10$ is a perfect square.
\end{proposition}

\begin{proof}
For nine general points let $C$ be their smooth cubic. A degree-$q$ curve
with multiplicity at least $m$ at all nine points has $q\ge3m$.
Indeed, if $q<3m$, B\'ezout forces $C$ to be a component; removing it
reduces $(q,m)$ to $(q-3,m-1)$ and leads by induction to a negative degree.
For the criterion put $m=q-d+a+1$. The inequality $q\ge3m$ implies $q<2d$,
and
\[
 \binom{q+2}{2}\ge\binom{3m+2}{2}
  =9\binom{m+1}{2}+1>9\binom{m-a+1}{2}.
\]
Thus \eqref{eq:two-term} fails.

Under the stated inequality for $r\ge10$, nonemptiness gives
$q\ge m\sqrt r>2m$, again implying $q<2d$. If $a\ge1$, then
\[
 \binom{q+2}{2}>\frac{q^2}{2}\ge\frac{rm^2}{2}
                  >r\binom{m-a+1}{2}.
\]
The last assertion uses Nagata's theorem for square numbers of points
\cite{Nag59}; see also \cite{CHMR13}.
\end{proof}

The proposition excludes only the comparison with a nonempty uniform
fat-point system in a degree where the predicted series has vanished.
It says nothing about other sources of defect, or about an excess in a
degree where the predicted Hilbert function is positive.

\section{The finite boundary and further questions}\label{sec:boundary}

We separate the geometric part of Theorem~\ref{thm:first} from its finite
verification. This also isolates precisely what the computation proves.

\begin{proof}[Proof of Theorem~\ref{thm:first}]
If $F_i=L_i^{13}M_i$, choose a nonzero conic $C$ through the five coefficient
points. Its twelfth power has degree $24$ and multiplicity at least $12$
at each point. Proposition~\ref{prop:comparison} puts $C^{12}$ in
$(S/I)_{24}^\vee$. Meanwhile
\[
 c_{3,14,5}(24)=\binom{26}{2}-5\binom{12}{2}=325-330=-5.
\]
This proves the failure for every choice of the factors.

Appendix~\ref{app:certificate} gives a fixed integer specialization for
which, modulo $32003$, the maps \eqref{eq:multiplication} have ranks
\[
 \rank\Phi_{23}=275,\qquad \rank\Phi_{24}=324,
 \qquad \rank\Phi_{25}=351.
\]
Nonzero minors lift to characteristic zero. The first rank excludes any
syzygy in degree at most $23$; the last makes the quotient zero in every
degree at least $25$. The middle rank, together with the conic lower
bound, gives dimension exactly one in degree $24$. Semicontinuity proves
the claimed generic Hilbert series.

The same certificate proves $G_{3,d,5}$ for every two-part partition
$(d-a,a)$ with $2\le d\le13$ and $1\le a\le\lfloor d/2\rfloor$, and for
$d=14$, $2\le a\le7$. Proposition~\ref{prop:coarsening} then gives every
non-pure partition in degrees at most $13$. A non-pure partition of $14$
other than $(13,1)$ has a subset of parts with sum between $2$ and $12$,
so it is covered by one of the positive two-part cases as well.
\end{proof}

Corollary~\ref{cor:conic} gives a failure for $(d-1,1)$ for every $d\ge14$.
By Theorem~\ref{thm:persistence}, the excess eventually occupies an interval
of degrees and becomes arbitrarily large.

There are two natural questions beyond these lower bounds. The first asks
whether the elementary comparison already accounts for all the defect in
the simplest family.

\begin{question}\label{q:exact}
For five general forms $L_i^{d-1}M_i\in\CC[x,y,z]_d$ and $q\ge d$, is
\[
 \HF_{S/(L_1^{d-1}M_1,\ldots,L_5^{d-1}M_5)}(q)
 =\max\bigl\{[t^q]G_{3,d,5},\ h_Z(q;q-d+2)\bigr\}?
\]
\end{question}

Both entries on the right are lower bounds: the first follows from
Anick's theorem and semicontinuity, and the second from
Proposition~\ref{prop:comparison}. Theorem~\ref{thm:first} determines the
first exceptional Hilbert series but does not resolve this question in
arbitrary degree.

The second question concerns the size of the residual factor. For five
general ternary generators, $(3-\sqrt5)/8$ is the asymptotic boundary of
the uniform fat-point criterion. Does the actual Hilbert-function defect
persist for larger proportions $a/d$? Determining this would distinguish
the obstruction studied here from effects of the complete factorization.

\appendix
\section{The exact finite certificate}\label{app:certificate}

We give the data and the reduction that make the finite verification
reproducible. Its only role is the sharpness statement and the exact
Hilbert series in Theorem~\ref{thm:first}.

For $d\ge2$ let
\[
 Q_d=\min\left\{q\ge d:\binom{q+2}{2}
                            -5\binom{q-d+2}{2}\le0\right\}.
\]
We have $Q_d<2d$, since the difference at $q=2d-1$ is
$-d(d+3)/2$. Furthermore the ratio
$\binom{q+2}{2}/\binom{q-d+2}{2}$ decreases with $q\ge d$.
Thus $Q_d$ is the first zero degree of $G_{3,d,5}$.

\begin{lemma}\label{lem:boundarychecks}
For five forms of degree $d$, injectivity of $\Phi_{Q_d-1}$ and
surjectivity of $\Phi_{Q_d}$ imply Hilbert series $G_{3,d,5}$.
\end{lemma}

\begin{proof}
A nonzero homogeneous syzygy in an earlier degree can be multiplied by
a nonzero monomial to produce one in degree $Q_d-1$. Hence there are no
syzygies up to that degree, and the Hilbert function there is the
difference of the source and target dimensions. If $I_{Q_d}=S_{Q_d}$,
then multiplication by variables gives $I_q=S_q$ for every $q\ge Q_d$.
\end{proof}

Use the following coefficient vectors over $\mathbb Z$, and form
$F_i=L_i^{d-a}M_i^a$:
\[
\begin{array}{c|c}
L_i&M_i\\\hline
(0,7,2)&(9,8,5)\\
(2,6,4)&(2,10,7)\\
(0,8,0)&(2,0,6)\\
(8,1,7)&(8,7,0)\\
(7,9,6)&(4,6,9)
\end{array}
\]
Rows of $\Phi_q$ are indexed by degree-$q$ monomials; columns are the
coefficient vectors of all $AF_i$ for degree-$(q-d)$ monomials $A$.
The file \texttt{verify\_manuscript.py} constructs these matrices and
performs Gaussian elimination modulo the prime $32003$. No random data
or floating-point rank decisions are used.

For every $2\le d\le13$ and $1\le a\le\lfloor d/2\rfloor$ the two
boundary ranks are maximal. The same holds for $d=14$, $2\le a\le7$.
The cutoffs are
\[
\begin{array}{c|rrrrrrrrrrrrr}
d&2&3&4&5&6&7&8&9&10&11&12&13&14\\\hline
Q_d&3&4&6&8&10&12&13&15&17&19&21&23&24.
\end{array}
\]
There are $96$ maximal-rank checks for these positive cases. The three
additional checks for $(d,a)=(14,1)$ are
\[
\begin{array}{c|rrr}
q&23&24&25\\\hline
\text{number of rows}&300&325&351\\
\text{number of columns}&275&330&390\\
\text{rank modulo }32003&275&324&351.
\end{array}
\]
The file \texttt{verification/rank\_checks.csv} records all $99$ ranks and
matrix sizes. The verifier runs with the Python standard library; optional
NumPy acceleration gives the same results. Both implementations have been
checked on all $99$ matrices. Use \texttt{--pure-python} to disable acceleration.

A nonzero minor modulo $32003$ remains nonzero over $\CC$, so the
maximal-rank checks prove generic characteristic-zero statements.
For the exceptional rank $324$, the conic argument supplies the matching
upper bound: reduction modulo a prime alone would not certify a
characteristic-zero rank defect.

\section*{Acknowledgements}
The author thanks Ralf Fr\"oberg for many discussions of these questions
and for their long collaboration.

\section*{Use of AI tools}
AI tools, including ChatGPT/Codex and Claude, were used to assist with
mathematical exploration, literature searches, computational verification,
and revision of the exposition. The author is responsible for the final
verification and presentation of the results.

\end{document}